\documentclass[12pt]{article}

\usepackage{latexsym,amsfonts,amsmath,amssymb,epsfig,tabularx,amsthm,dsfont,mathrsfs}
\usepackage{graphicx}
\usepackage{hyperref}
\usepackage{enumerate}
\usepackage{color}
\usepackage{enumitem}

\allowdisplaybreaks

\theoremstyle{definition}

\newtheorem{thm}{Theorem}[section]
\newtheorem{rem}[thm]{Remark}

\newtheorem{lem}[thm]{Lemma}
\newtheorem{prop}[thm]{Proposition}
\newtheorem{cor}[thm]{Corollary}
\newtheorem{example}[thm]{Example}
\newtheorem{ass}[thm]{Assumption}

\renewcommand{\d}{{\rm d}}

\newcommand{\norm}[1]{\left\Vert #1 \right\Vert}

\newcommand{\N}{\mathbb{N}}
\newcommand{\E}{\mathbb{E}}
\newcommand{\R}{\mathbb{R}}

\newcommand{\cX}{\mathcal{X}}
\newcommand{\cY}{\mathcal{Y}}
\newcommand{\cF}{\mathcal{F}}
\newcommand{\cA}{\mathcal{A}}

\DeclareMathOperator{\eps}{\varepsilon}

\title{Almost sure convergence rates of adaptive increasingly rare Markov chain Monte Carlo}

\author{Julian Hofstadler$^1$ \and Krzysztof {{\L}atuszy\'{n}ski}$^2$  \and Gareth O. Roberts$^2$ \and Daniel Rudolf$^3$} 

\date{%
	$^1$University of Bath\\%
	$^2$University of Warwick\\%
	$^3$University of Passau\\[2ex]%
	\today
}
\begin{document}
\maketitle	
\begin{abstract}
We consider adaptive increasingly rare Markov chain Monte Carlo (MCMC) algorithms, which are adaptive MCMC methods, where the adaptation concerning the ``past'' happens less and less frequently over time. Under a contraction assumption with respect to a Wasserstein-like function we deduce upper bounds of the convergence rate of Monte Carlo sums taking a renormalisation factor into account that is ``almost'' the one that appears in a law of the iterated logarithm. We demonstrate the applicability of our results by considering different settings, among which are those of simultaneous geometric and uniform ergodicity. All proofs are carried out on an augmented state space, including the classical non-augmented setting as a special case. In contrast to other adaptive MCMC limit theory, some technical assumptions, like diminishing adaptation, are not needed.
\end{abstract}


\textbf{Classification:} 65C05; 65C20; 60J22 
\section{Introduction}

Given some probability space $(\cX, \mathcal{F}_\cX, \nu)$ an omnipresent challenge in computational statistics is the approximation of expectations  
\begin{equation}\label{equ:target_expectation}
	\nu(f) =
	\int_\cX f(x) \nu( \d x).
\end{equation}
Here $\nu $ is a desired distribution and $f \colon \cX \to \R$ a suitable integrand of interest. 
It is a common situation that sampling from $\nu$ directly is not possible or computationally infeasible, for instance, since the density of $\nu$ might contain an unknown normalising constant. 

\textit{Markov chain Monte Carlo} (MCMC) methods provide a powerful tool for approximating expectations as given in \eqref{equ:target_expectation}, especially if direct sampling is no option.
The main idea is to construct a Markov chain\footnote{{All appearing random variables are assumed to be defined on a common probability space $(\Omega,\mathcal{F},\mathbb{P})$.}} $(X_n)_{n \in \N_0}$ which admits $\nu$ as limit distribution and 
under rather mild conditions one can show that the Monte Carlo sums\footnote{To ease notation later we do not use the initial point $X_0$ in the Monte Carlo sum.} 
$S_n f = \frac{1}{n}\sum_{j=1}^{n}f(X_j)$
converge to $\nu(f)$ (in some sense), provided that $f\colon \cX \to \R$ is measurable and $\nu(f)$ is finite, cf. \cite{meyn2012markov,douc2018markov}.
Furthermore, explicit error bounds \cite{rudolf2012explicit,latuszynski2015nonasymptotic} as well as concentration inequalities \cite{glynn2002hoeffding,MIASOJEDOW2014hoeffding} are available.

Adaptive MCMC allows to update the single step transition mechanism by taking the whole history into account and therefore leads to non-Markovian processes (in general). 
Numerical experiments demonstrate that if the adaptation is 
	properly done, then
	convergence to the correct limit distribution can be improved.
However, a careless implementation of the adjustment might prevent the algorithm from achieving desirable convergence properties, we refer to \cite[Section 2]{atchade2005adaptive} as well as \cite{latuszynski2014adapfail}.    

The non-Markovian behaviour of adaptive MCMC algorithms leads to rather complicated proof techniques in the literature. 
To ease the mathematical analysis the authors of \cite{chimisov2018air} suggest to use adaptive increasingly rare (AIR) algorithms, where adaptation takes place less and less frequently.  
Indeed, certain technical requirements can be dropped in the theory of those methods.
Moreover, numerical studies, cf. \cite[Section 2.2]{chimisov2018air}, show that AIR versions of  algorithms can perform equally well compared to their purely adaptive counterparts, while saving computational resources.

The goal of this paper is to study the convergence properties of renormalised Monte Carlo sums corresponding to an AIR scheme $(X_n)_{n\in\N_0}$.
Setting $r(n) = \sqrt{n}(\log(n))^{1/2+\eps}$, or $r(n) = n^{1/2 + \eps}$,
our main results, see Section~\ref{sec:Wasserstein_theory}, show 
\begin{equation}\label{equ:main_result_description}
	\lim_{n \to \infty} \left\vert \frac{1}{r(n)} \sum_{j=1}^n\left(  f(X_j) - \nu(f)\right)  \right\vert=0, 
\end{equation}
almost surely, given suitable $\eps>0$ and $f \colon \cX \to \R$.
The main assumption in Section \ref{sec:Wasserstein_theory} is a contraction condition for a Wasserstein-like function. 
We apply our results to different settings in Section \ref{sec:examples} leading to convergence rates for $S_nf$ based on AIR MCMC methods.  
To be more precise, based upon \eqref{equ:main_result_description} it is possible to quantify the path-wise behaviour of $S_n f$, see Corollaries \ref{cor:convergence_speed_uniform_ergodicity} and~\ref{cor:rate_geometric_ergodic}. These show that for any $n \in \N$ and almost all $\omega\in \Omega$ we have 
\begin{equation}\label{equ:almost_sure_bound}
	\left\vert  \frac{1}{n}\sum_{j=1}^n f(X_j(\omega)) - \nu(f)   \right\vert \leq \frac{{C(\omega)}r(n)}{n},
\end{equation}
where $C \colon \Omega \to (0, \infty)$, and $r(n)$ is as in \eqref{equ:main_result_description}. 
Taking into account the choices for $r(n)$ above, the right hand side in \eqref{equ:almost_sure_bound} decreases almost like $1/\sqrt{n}$ if $\eps$ is small.  
These bounds may be practically relevant since they provide information about the speed of convergence of $S_n f$ to $\nu(f)$, even in cases where we only simulate $(X_n)_{n \in \N_0}$ once.

The rate $r(n)$ in \eqref{equ:main_result_description} involves some number $\eps>0$. 
As is discussed in Section \ref{sec:examples} this parameter depends on  
\begin{itemize}
	\item[1)] the size of the windows between consecutive adaptations, and
	\item[2)] the complexity of the integrand under consideration.
\end{itemize}

The main idea of the proofs is to use a martingale approximation based on solving Poisson's equation, which is known to be a powerful tool, we refer for instance to \cite{andrieu2006ergodicity,saksman2010ergodicity,atchade2013Bayesian}.
See also \cite{atchade2010limit,atchade2012limit} where a generalised Poisson's equation was used to derive ergodicity results, strong laws of large numbers, and central limit theorems for sub-geometrically ergodic kernels.
While the martingale part is typically easy, there is also a remainder part, which for purely adaptive schemes can be tricky to control.
Usually one requires technical assumptions, however, 
for some algorithms, e.g. adaptive Metropolis \cite{haario2001adaptive}, we know these conditions are satisfied and limit theory is available \cite{andrieu2006ergodicity,saksman2010ergodicity}.

In this paper we focus on AIR versions of algorithms. 
An important observation is that for this class of methods most of the (difficult) terms in the martingale approximation cancel. 
Hence some rather technical (necessary) conditions (cf. \cite{atchade2005adaptive,laitinen2024invitationadaptivemarkovchain}), e.g. \textit{diminishing adaptation}, are not required to show that the limit in \eqref{equ:main_result_description} is almost surely zero.

Our analysis of AIR MCMC is done in an augmented state space setting, including the classical (non augmented) regime as special case.
This is motivated by the recent work \cite{Pompe_2020Framework}, where it is shown that moving to an augmented state space allows to tackle difficulties like multimodality of the target $\nu$.

Let us comment on how our results fit into the existing literature.
Different results for the ergodicity of adaptive MCMC methods in general and the corresponding Monte Carlo sums are available. 
There are laws of large numbers \cite{haario2001adaptive,atchade2005adaptive, andrieu2006ergodicity,roberts2007coupling,saksman2010ergodicity,fort2011convergence,Pompe_2020Framework}, central limit theorems \cite{andrieu2006ergodicity,saksman2010ergodicity} as well as results about convergence to the correct target measure \cite{roberts2007coupling,fort2011convergence}.
In \cite{chimisov2018air} the convergence of AIR MCMC schemes is provided and,
moreover,  the authors of \cite{andrieu2007efficiency} studied the efficiency of adaptive MCMC compared to an ``optimal'' MCMC method.

Moreover, in \cite{andrieu2007efficiency,atchade2015convergence} quantitative results concerning the ergodicity of an adaptive chain are deduced.  
It is shown there that for different algorithms one has $\norm{\mathbb{P}(X_n \in \cdot) - \nu}_{tv} \leq C/n$, where $\norm{\cdot}_{tv}$ denotes the 
total variation distance.
Bounds for the mean squared error of traditional and AIR algorithms are 
reported in
\cite{atchade2015convergence, chimisov2018air}.
In particular, both papers show that, in different settings, rates of order $1/n$ are possible.  
In contrast to that, our bounds in \eqref{equ:almost_sure_bound} are understood path-wise. 
For small $\eps$ the rate $r(n)$ in \eqref{equ:main_result_description} is already close to that one in the law of the iterated logarithm, which, at least for the special case where $(X_n)_{n \in \N}$ consists of i.i.d. 
random variables
with $X_1 \sim \nu$, cannot be improved. 

Limit theory is usually proven under uniform or geometric convergence assumptions and we are not aware of results that hold in a Wasserstein contraction setting.
Furthermore, convergence of renormalised Monte Carlo sums as in \eqref{equ:main_result_description} (or also laws of iterated logarithm) are to the best of our knowledge not known for adaptive MCMC methods.  
It is however worth mentioning that \cite[Proposition 6]{saksman2010ergodicity} could be used to obtain similar results for stochastic approximation type algorithms in a $V$-simultaneously geometrically ergodic setting, see the discussion at the end of Section \ref{sec:Wasserstein_theory} and Example~\ref{ex:Metropolis_geometric_ergodic}.

The goal of the present paper is to study path-wise convergence in a more general setting:
Under a contraction condition 
with respect to (w.r.t.)
Wasserstein-like function we provide results describing the convergence behaviour of \eqref{equ:main_result_description}, see Theorems \ref{thm:limit_result_bounded_eccentricity_large_beta}, \ref{thm:limit_result_bounded_eccentricity_small_beta}, \ref{thm:limit_result_lyapunov}.  
These are applied to different examples, {including the simultaneous geometric and uniform ergodicity cases}, which are of interest in {the} adaptive  MCMC literature. 

The outline of the paper is as follows. In Section \ref{sec:preliminaries} we introduce some notation and general concepts which are used throughout the paper. 
Theoretical results are stated and proven in Section \ref{sec:Wasserstein_theory} and applied to several examples in Section \ref{sec:examples}. 
The appendix contains some auxiliary material.

\section{Preliminaries}\label{sec:preliminaries}

In this section we provide all required objects to establish our limit theorems later. 
We start with some general notation and assumptions. 

Following \cite{Pompe_2020Framework}, we consider an augmented state space, denoted by $\cY = \cX \times \Phi$, where $(\Phi, \mathcal{F}_\Phi)$ is some ``auxiliary'' measurable space and $\mathcal{F}_{\cY} = \mathcal{F}_\cX \otimes \mathcal{F}_\Phi$. 
To avoid problems concerning measurability, all involved spaces are assumed to be Polish with countably generated $\sigma$-algebra.  
Besides that, there are no particular requirements on $(\Phi, \mathcal{F}_\Phi)$ and, moreover, different choices for $\mathcal{F}_\cY$ may be feasible as well.
Note that choosing $\Phi $ to be a singleton corresponds to the classical non-augmented case.  

Given some measurable space $(\mathcal{I}, \mathcal{F}_{\mathcal{I}})$ we consider 
\begin{itemize}
	\item[1)] a family of probability distributions $\{{\pi}_\gamma\}_{\gamma \in \mathcal{I}}$ on $(\cY, \mathcal{F}_{\cY})$, and
	\item[2)] a family of transition kernels $\{ {P}_\gamma \}_{\gamma \in \mathcal{I}}$ on $(\cY, \mathcal{F}_{\cY})$.
\end{itemize}

To be more precise, by 2) we mean that for any $\gamma \in \mathcal{I}$ the mapping ${P}_\gamma \colon \cY \times \mathcal{F}_{\cY} \to [0,1]$ is a transition kernel. We assume that ${\pi}_\gamma$ has marginal $\nu$ for any $\gamma \in \mathcal{I}$, that is, ${\pi_\gamma}(B \times \Phi) = \nu(B)$ for any $B \in \mathcal{F}_\cX$, and is the invariant distribution of ${P}_\gamma({y}, \cdot )$.
Moreover, we assume that any kernel $P_\gamma$ is $\psi$-irreducible and aperiodic, see \cite{meyn2012markov} for definitions, and that the mapping $(y,\gamma,A) \mapsto P_\gamma(y, A)$ is a transition kernel on $(\cY \times \mathcal{I}, \mathcal{F}_\cY \otimes \mathcal{F}_{\mathcal{I}})$.

	Let us briefly comment on the role of the auxiliary space $\Phi$. 
	As the name suggests, it contains ``auxiliary variables'' which, in some form, should be useful for the computation of $\nu(f)$. 
	For instance, for multimodal $\nu$, with modes $x_1 , \dots , x_T$, the authors of \cite{Pompe_2020Framework} set $\Phi= \{1, \dots, T\}$, and $\nu$ is targeted by an algorithm performing ``local'' and ``jump'' moves: 
	Given $(x,i) \in \cX \times \Phi$ as current point, a local move keeps $i \in \Phi$ fixed and only updates $x$, i.e. we target the current mode. 
	A jump move on the other hand keeps $x$ fixed and allows to target a different mode afterwards, cf. \cite[Section 2]{Pompe_2020Framework}. 

\subsection{Adaptive and AIR MCMC}\label{subsec:AIR_MCMC} 
We consider an adaptive MCMC chain, say $(Y_n, \Gamma_n)_{n \in \N_0}$, evolving in the space $(\cY, \mathcal{I})$, where $(\Gamma_n)_{n \in \N_0}$ denotes the sequence of (random) adaptation parameters. 
In other words, $\cY$ is the state space of $(Y_n)_{n \in \N_0}$ and $\mathcal{I}$ that of $(\Gamma_n)_{n \in \N_0}$.  
We assume that there exists some underlying (sufficiently regular) probability space $(\Omega, \mathcal{F}, \mathbb{P})$ and set $\mathcal{F}_{n} = \sigma\left(Y_0 , \dots, Y_n , \Gamma_0 , \dots, \Gamma_n \right)$, such that $(\mathcal{F}_n)_{n \in \N}$ is the natural filtration of $(Y_n, \Gamma_n)_{n\in \N_0}$.
We assume that $(Y_0, \Gamma_0) \equiv (y_0, \gamma_0 )$ for some initial values $y_0 \in \cY$ and $\gamma_0 \in \mathcal{I}$.

The one step transition probabilities of $(Y_n )_{n \in \N_0}$ are specified via 
\begin{equation}\label{equ:adaptive_MCMC_dynamics}
	\mathbb{P}\left[Y_{n+1} \in B \vert Y_n = {y}, \Gamma_n= \gamma \right] = {P}_\gamma(y, B),
\end{equation}
where $B \in \cF_{\cY}$. 
Moreover, the adaptation parameter $\Gamma_n$ is chosen according to some update rule\footnote{The functions $U_n$ are assumed to be measurable.} $U_n \colon \cY^{n+1} \times \mathcal{I}^n \to \mathcal{I}$, allowing to ``improve'' the kernel of $Y_{n+1}$ based on the already known values of $Y_0, \dots, Y_n$ and $\Gamma_0, \dots, \Gamma_{n-1}$. 
That is $\Gamma_n = U_n (Y_0, \dots, Y_n , \Gamma_0 , \dots,\Gamma_{n-1})$.

Using the regularity of all involved spaces and standard techniques of measure theory we have that $\mathbb{P}$-almost surely
\begin{equation}\label{equ:conditional_expectation}
	\E \left[g (Y_{n+1}) \vert \cF_{n}\right] = P_{\Gamma_n} g (Y_n),
\end{equation}
where $g \colon \cY \to \R$ is an arbitrary measurable function satisfying $P_\gamma g(y) < \infty$ for any $(y,\gamma)\in\cY\times\mathcal{I}$.

We assume that for any $n \in \N$ the mapping $\omega \mapsto \pi_{\Gamma_n(\omega)}(g)$ is measurable for any (measurable) function $g \colon \cY \to \R$ such that $\pi_\gamma (\vert g\vert ) < \infty$ for all $\gamma \in \mathcal{I}$.

Since $(Y_n)_{n \in \N_0}$ evolves in the augmented space $\cY = \cX \times \Phi$, each state has the form $Y_n = (X_n, \phi_n)$, with an $\cX$-valued component $X_n$ and a $\Phi$-valued component $\phi_n$.
Given some measurable function of interest $f \colon \cX \to \R$ with $\nu(\vert f \vert) < \infty $, we set $h= h_f = f \cdot \mathds{1}_\Phi$. 
Using the fact that $\pi_\gamma$ has marginal $\nu$ we obtain the following 
\begin{equation}\label{equ:S_n_on_augmented_space}
	S_n f - \nu(f) = \frac{1}{n}\sum_{j=1}^n f(X_j) - \nu(f)= \frac{1}{n}\sum_{j=1}^n \left( h(Y_j) - {\pi}_{\Gamma_{j-1}}(h)\right),
\end{equation}
holding $\mathbb{P}$-a.s. for any $n\in \N$. 

In this paper we focus on AIR MCMC algorithms, as introduced in \cite{chimisov2018air}.
This class of methods allows to update $\Gamma_n$ only at specific times, in contrast to purely adaptive schemes, where $\Gamma_n$ may change after each step.
Given $\beta>0$ we set $k_j = \lceil j^\beta \rceil$ and allow adaptation only at the times 
\[
T_m = \sum_{j=1}^m k_j.
\]
That is, for $i_1, i_2 \in \{ T_j , \dots, T_{j+1}-1\}$ (with $j \in \N$) we have $\Gamma_{i_1} = \Gamma_{i_2}$.  

A simple, yet important observation is that for any $\beta>0$ there exist constants $c_\beta, C_\beta \in \R$ such that 
\begin{equation}\label{equ:AIR_MCMC_steps_bound}
	c_\beta m^{1+\beta}  \leq 
	T_m \leq C_\beta m^{1+\beta}.
\end{equation}
In particular, this means that we adapt less and less frequently.

Let us comment on the possible choices for $\beta \in (0, \infty)$. The authors of \cite{chimisov2018air} compared the performance of	AIR and traditional adaptive random walk Metropolis algorithm in numerical experiments, see Section 2.2 of the mentioned paper. The target distribution in these simulations was 100 dimensional, and for $\beta \in [1,2]$ the authors of \cite{chimisov2018air} report that AIR algorithms perform similarly well as purely adaptive methods.

\subsection{Martingale approximation}
Our results shall be proven using a martingale approximation technique based on solutions of Poisson's equation.  
Providing a powerful tool for the analysis of general state space Markov chains, see e.g. \cite{meyn2012markov,douc2018markov}, Poisson's equation was later also used in the adaptive setting, e.g. in \cite{andrieu2006ergodicity,saksman2010ergodicity,atchade2013Bayesian}. See also \cite{atchade2010limit,atchade2012limit} for further generalisations.

For a function of interest $h \colon \cY \to \R$ as well as $\gamma \in \mathcal{I}$ we say $u_\gamma\colon \cY \to \R$ is a solution of Poisson's equation if
\[
u_\gamma (y) -P_\gamma u_\gamma(y) 
= 
h(y) - {\pi}_\gamma(h)
\]
for any $y \in \cY$.
Provided such an $u_\gamma$ exists for any $\gamma \in \mathcal{I}$ we obtain the following decomposition for the Monte Carlo sum of an adaptive MCMC method driven by \eqref{equ:adaptive_MCMC_dynamics}:
\begin{align}\label{equ:martingale_decomposition_general}
	\sum_{j=1}^n \left( h(Y_j)- {\pi}_{\Gamma_{j-1}}(h) \right) 
	&= 
	\sum_{j=1}^n \left( {u}_{\Gamma_{j-1}}(Y_j)- {P}_{\Gamma_{j-1}}{u}_{\Gamma_{j-1}}(Y_{j-1}) \right)\nonumber \\
	&\qquad+ \sum_{j=1}^n \left( {u}_{\Gamma_{j}}(Y_j)- {u}_{\Gamma_{j-1}}(Y_{j}) \right) \nonumber\\
	&\qquad+h(Y_n) - h(Y_0) 
	+{u}_{\Gamma_0}(Y_0) - {u}_{\Gamma_{n}}(Y_n).
\end{align}

Recall that for an AIR scheme adaptation only occurs at the times $\{T_j\}_{j \in \N_0}$ with $T_0 := 0$. 
Assume $T_m \leq n+1 < T_{m+1}$,
then, since $u_{\Gamma_{i+1}}(Y_{i+1}) - u_{\Gamma_{i}}(Y_{i+1}) =0$ for any $i \in \{T_{j}, \dots, T_{j+1}-2\}$, equation \eqref{equ:martingale_decomposition_general} simplifies to 
\begin{align}\label{equ:martingale_decomposition_air_MCMC_general}
	\sum_{j=1}^n \left( h(Y_j)- {\pi}_{\Gamma_{j-1}}(h) \right) 
	&=
	\sum_{j=1}^n \left( {u}_{\Gamma_{j-1}}(Y_j)- {P}_{\Gamma_{j-1}}{u}_{\Gamma_{j-1}}(Y_{j-1}) \right)\nonumber \\
	&\qquad+ \sum_{j=0}^{m-1} \left( u_{\Gamma_{T_{j+1}}}(Y_{T_{j+1}}) - u_{\Gamma_{T_j}}(Y_{T_{j+1}}) \right)  \nonumber\\
	&\qquad+h(Y_n) - h(Y_0) 
	+{u}_{\Gamma_0}(Y_0) - {u}_{\Gamma_{n}}(Y_n).
\end{align}

Here ``simplification'' refers to the fact that \eqref{equ:martingale_decomposition_air_MCMC_general} contains less summands of the form ${u}_{\Gamma_{j}}(Y_j)- {u}_{\Gamma_{j-1}}(Y_{j})$ compared to \eqref{equ:martingale_decomposition_general}.

Let us define the following quantities:
\begin{align*}
	M_n &= \sum_{j=1}^n \left( {u}_{\Gamma_{j-1}}(Y_j)- {P}_{\Gamma_{j-1}}{u}_{\Gamma_{j-1}}(Y_{j-1}) \right); \\
	\Delta_j &:= {u}_{\Gamma_{j-1}}(Y_j)- {P}_{\Gamma_{j-1}}{u}_{\Gamma_{j-1}}(Y_{j-1}); \\
	\text{and} \qquad R_m &:= \sum_{j=0}^{m-1} \left( u_{\Gamma_{T_{j+1}}}(Y_{T_{j+1}}) - u_{\Gamma_{T_j}}(Y_{T_{j+1}}) \right).
\end{align*} 
It is worth mentioning that, if $\E\vert M_n \vert <\infty$, then
	$\E [M_n\vert \mathcal{F}_{n-1}] = M_{n-1}$ almost surely.

\subsection{Distance-like and Wasserstein-like functions}
We consider a measurable space $(\mathcal{A}, \mathcal{F}_\mathcal{A})$, which is Polish w.r.t. some metric $d_\cA$ and has a countably generated $\sigma$-algebra. 

Following \cite{hairer2011asymptotic}, we call a function $d \colon \cA \times \cA \to \R^+$ \textit{distance-like} if: 
	\begin{itemize}
		\item[a)] For any $x,y\in \cA$ holds $d(x,y) = d(y,x)$. 
		\item[b)] We have $d(x,y) =0 \;\Longleftrightarrow \; x=y$.
		\item[c)] It is lower semi-continuous w.r.t. the product topology of $d_{\mathcal{A}}$.
	\end{itemize}
	
	A non-negative function which satisfies a) and b) is often\footnote{{The definition of a semi-metric seems to be used inconsistently in the literature, as some authors call $\widetilde{d} \colon \mathcal{A} \times \mathcal{A} \to \R^+$ a semi-metric if it satisfies a), triangle inequality and $\widetilde{d}(x,x) =0$ for any $x \in \mathcal{A}$ (instead of b)).}} called a \textit{semi-metric}. 
	Hence, a distance-like function is a semi-metric which additionally has to be lower semi-continuous (w.r.t. $d_{\mathcal{A}}$).

Given some distance-like function $d$ on $\cA$ and probability measures $\mu_1, \mu_2$ on $(\cA, \mathcal{F}_\cA)$ we define
\[
\mathcal{W}(\mu_1, \mu_2):=\inf_{\xi \in \mathcal{C}(\mu_1, \mu_2)} \int_{\cA^2} d(x,y) \xi(\d x , \d y) ,
\]
where $\mathcal{C}(\mu_1, \mu_2)$ is the set of all couplings of $\mu_1, \mu_2$.
As argued in \cite{hairer2011asymptotic} below Definition 4.3, $\mathcal{W}$ is a reasonable tool for quantifying the difference between probability measures.
Following \cite{hosseini2023spectral} we call $\mathcal{W}$ a \textit{Wasserstein-like} function (w.r.t. $d$).
Note, if $d$ is a metric, then $\mathcal{W}$ corresponds to the usual $1$-Wasserstein distance.

A function $\psi \colon \cA \to \R$ is called Lipschitz w.r.t. the distance-like function $d$ if there exists some constant $L \in (0, \infty)$ such that  
	\[
	\vert \psi(x) - \psi(y) \vert \leq L \cdot  d(x,y),
	\]
	for any $x,y \in \cA$.
We write $\text{Lip}_1$ for the set of all functions that are Lipschitz (w.r.t. $d$) with constant $1$. 
Exploiting lower semi-continuity of $d$ and that the underlying space is Polish, we obtain the following important inequality.
\begin{lem}\label{lem:duality_inequality}
	For $\mathcal{W}$ w.r.t. distance-like function $d$ we have
	\[
	\mathcal{W}(\mu_1, \mu_2) \geq \sup_{\psi \in \text{Lip}_1} \left\vert  \int_\cA \psi(x) \mu_1(\d x) - \int_\cA \psi(y) \mu_2(\d y)  \right\vert.
	\]
\end{lem}
\begin{proof}
	Let $\psi \in \text{Lip}_1$. 
	By lower semi-continuity of $d$ we have that there exists an optimal coupling $\xi$, i.e. $\mathcal{W}(\mu_1,\mu_2) =  \int_\cA d(x,y) \xi(\d x, \d y)$.  
	Thus 
	\begin{align*}
		\left\vert \int_\cA \psi(x) \mu_1(\d x) - \int_\cA \psi(y) \mu_2(\d y)  \right\vert
		=
		\left\vert\int_\cA\int_\cA \psi(x)  -  \psi(y) \xi(\d x ,\d y)  \right\vert\\
		\leq 
		\int_\cA\int_\cA d(x,y) \xi(\d x,\d y) = \mathcal{W}(\mu_1, \mu_2).  
	\end{align*}
\end{proof}
The well-known Kantorovich-Rubinstein duality states that if $d$ is a (lower semi-continuous) metric, then one would have equality in Lemma \ref{lem:duality_inequality}.

\section{Convergence rates via Wasserstein contraction}\label{sec:Wasserstein_theory}
In this section we provide limit theorems based on a contraction assumption for a Wasserstein-like function. 

Recall that $\cY$ is Polish w.r.t. to some metric $d_\cY$ and let $d \colon \cY \times \cY \to \R^+$ be a distance-like function.
For $\gamma \in \mathcal{I}$ we define the following quantities: 
\begin{itemize}
	\item[1)] The eccentricity (see \cite{Joulin2010Curvature}) is defined as 
	\[
	E_\gamma (y) = \int_\cY d(y, y' ) \pi_\gamma(\d y').
	\] 
	\item[2)] Given $\ell \in \N$, the Wasserstein contraction coefficient\footnote{For the sake of brevity we refer to this as Wasserstein contraction coefficient, even if $\mathcal{W}$ is ``only'' a Wasserstein-like function.} for $P_\gamma^\ell$ is 
	\[
	\tau(P_\gamma^\ell) = \sup_{\substack{y_1, y_2 \in \cY\\y_1 \neq y_2}} \frac{\mathcal{W} (P_\gamma^\ell (y_1, \cdot ), P_\gamma^\ell(y_2, \cdot))}{d(y_1, y_2)}.
	\]
\end{itemize}

Reviewing the proofs of \cite[Proposition 14.3 and Proposition 14.4]{dobrushin96} we see that also for $\mathcal{W}$ based on a distance-like function one has 
\begin{itemize}
	\item[i)] submultiplicativity, that is $\tau(P_\gamma P_{\gamma'}) \leq \tau(P_\gamma) \tau(P_{\gamma'})$, and,
	\item[ii)] contractivity, that is $\mathcal{W}(\mu_1 P_\gamma , \mu_2 P_\gamma) \leq \tau(P_\gamma) \mathcal{W}(\mu_1, \mu_2)$,
\end{itemize}
for any $\gamma, \gamma'\in \mathcal{I}$ and any probability measures $\mu_1, \mu_2$ on $(\cY , \mathcal{F}_\cY$).

The following contraction condition is our main assumption. 
\begin{ass}\label{assumption_main}
	For each $\gamma \in \mathcal{I}$ we assume that $\pi_\gamma$ is the invariant distribution of $P_\gamma$ and that 
	\[
	\tau(P_\gamma) \leq M \qquad \text{and} \qquad
	\tau(P_\gamma^{k_0}) \leq \tau,
	\]
	with $M \in [1, \infty)$, $\tau \in [0,1)$ and $k_0 \in \N$ independent of $\gamma$.
\end{ass}

	Assumption \ref{assumption_main} formulates a (uniform) contraction property w.r.t. a Wasserstein-like function.
	As we elaborate in Section \ref{sec:examples}, a suitable choice of the distance-like function $d$ allows us to work within the simultaneously uniformly ergodic, or geometric ergodic setting, which are classical in the adaptive MCMC literature.
	Hence the Wasserstein contraction assumption can be seen as a more general framework, which allows to recover classical examples as special cases. 
	
	Assumption \ref{assumption_main} requires all bounds to be uniform in $\gamma$, which is somewhat restrictive. 
	Yet, in Section \ref{sec:examples} we provide examples where it is satisfied. 
Dropping the uniformity may lead to algorithms which do no longer obey a law of large numbers, cf. Appendix~\ref{sec:appendix_counter_example}. 
		Loosely formulated, Assumption~\ref{assumption_main} prevents us from using 
		``arbitrarily bad'' kernels. However, under
	additional conditions it may be possible to weaken Assumption \ref{assumption_main}. 
	For instance, one could consider stochastic approximation type algorithms, see \cite{andrieu2006ergodicity,saksman2010ergodicity}, where, roughly speaking, the parameter sequence $(\Gamma_n)_{ n \in \N_0}$ evolves along subsets $\mathcal{K}_1 \subseteq \mathcal{K}_2 \subseteq \mathcal{K}_3 \subseteq \dots \subseteq \mathcal{I}$.
	Similar concepts appear also in different areas of MCMC theory, see e.g. \cite{Fort2003convergence}.
	In some sense, Assumption~\ref{assumption_main} allows to consider more general adaptation mechanisms, at the cost of having a stronger assumptions on the family of kernels $\{P_\gamma\}_{\gamma\in \mathcal{I}}$.

\begin{rem}
	In what follows expectations of the form  $\E [E_{\Gamma_{j}}(Y_k)]$, for some $j,k \in \N_0$, appear.
	Let us emphasize that for any $j, k \in \N_0$ the quantity $E_{\Gamma_{j}}(Y_k)$ is a well-defined, non-negative random variable, see  Appendix~\ref{sec:measurability_of_E_Gamma}. 
\end{rem}
\subsection{Solving Poisson's equation}
The upcoming result employs the ``Kantorovich-Rubinstein inequality'' of Lemma \ref{lem:duality_inequality}, to bound $P_\gamma^\ell f(x) - \pi_\gamma(f)$, where $\ell \in \N$ and $f$ is Lipschitz, by $\tau(P_\gamma^\ell)$. 
	This allows to solve Poisson's equation using Assumption \ref{assumption_main}, which is vital for the martingale approximation technique described in Section \ref{sec:preliminaries}. 

\begin{prop}\label{prop:poisson_equation_wasserstein_setting}
	Let $\gamma \in \mathcal{I}$, $f \colon \cY \to \R$ be $\pi_\gamma$ integrable and Lipschitz with constant $L \in (0, \infty)$ and let $E_\gamma (y) < \infty$ for any $y \in \cY$. 
	If Assumption~\ref{assumption_main} is true, then the function $u_\gamma\colon \cY \to \R$ given via 
	\[
	u_\gamma(y ) = \sum_{\ell=0}^\infty ( P_\gamma^\ell  f(y) - \pi_\gamma(f)) 
	\]   
	is well-defined and we have
	\begin{itemize}
		\item[i)] $\vert u_\gamma (y) \vert \leq \frac{k_0L M^{k_0} E_\gamma(y)}{1- \tau}$ for any $y \in \cY$,
		\item[ii)] $u_\gamma (y) - P_\gamma u(y) = f(y) - \pi_\gamma(f)$ for any $y \in \cY$.
	\end{itemize}
\end{prop}

\begin{proof}
	Let $\ell = tk_0 +r \in \N$ with $t \in \N_0$ and $0 \leq r < k_0$ as well as $y \in \cY$. 
	By Lemma \ref{lem:duality_inequality} and the properties of contraction coefficients   
	\begin{align*}
		\left\vert P_\gamma^\ell f(y) - \pi_\gamma(f)  \right\vert &\leq L \sup_{\psi \in \text{Lip}_1}\left\vert P^\ell \psi(y) - \pi_\gamma(\psi) \right\vert 
		\le L  \mathcal{W}(P_\gamma^{\ell}(y, \cdot), \pi_\gamma ) \\
		&\leq L \tau(P_\gamma^r) \tau(P_\gamma^{tk_0}) \mathcal{W}(\delta_y, \pi_\gamma)
		\leq L M^r \tau^{t} \mathcal{W}(\delta_y, \pi_\gamma) \\
		&\leq L M^{k_0} \tau^{t} E_\gamma(y) < \infty.
	\end{align*}
	Using this bound we obtain 
	\begin{align*}
		\vert u_\gamma(y )\vert  &\le \sum_{\ell=0}^\infty \vert P_\gamma^\ell f(y) - \pi_\gamma(f) \vert =
		\sum_{t=0}^\infty \sum_{j= tk_0 }^{(t+1) k_0-1 } \vert P_\gamma^{j} f(y) - \pi_\gamma(f) \vert\\
		&\leq \sum_{t=0}^\infty \sum_{j= tk_0 }^{(t+1)k_0 -1} L M^{k_0} \tau^{t} E_\gamma(y) 
		= \frac{k_0L M^{k_0} E_\gamma(y)}{1- \tau}.
	\end{align*}
	This shows that $u_\gamma$ is well-defined and also that i) holds.

	It remains to show ii): 
	It is no restriction to assume that $\pi_\gamma (f) =0$ (otherwise consider $g = f- \pi_\gamma(f)$). 
	For any $n \in \N$ we set 
	\[
	\psi_n^+ = \sum_{\ell=0}^n P^\ell_\gamma f^+  \qquad\text{and}\qquad \psi_n^- = \sum_{\ell=0}^n P_\gamma^\ell f^-,
	\]
	where $f^+$ and $f^-$ are the positive and negative part of $f$. 
	For any $n \in \N$  the definition of $\psi_n^+$ implies
	\[
	P_\gamma \psi_n^+ = \psi_{n+1}^+ - f^+ \qquad\text{and}\qquad P_\gamma \psi_n^- = \psi_{n+1}^- - f^-.
	\]
	By monotone convergence for any $y_0 \in \cY$, 
	\[
	\lim_{n \to \infty} \left( P_\gamma \psi_n^+ (y_0) \right) = P_\gamma \left( \lim_{n \to \infty} \psi_n^+\right)  (y_0).
	\]
	By the same arguments, the same statement for $\psi_n^-$ (instead of $\psi_n^+$) is deduced.
	Clearly, $u_\gamma = \lim_{n \to \infty} (\psi_n^+ - \psi_n^-)$ and hence by linearity 
	\begin{align*}
		P_\gamma u_\gamma = P_\gamma \left( \lim_{n \to \infty} (\psi_n^+ - \psi_n^-) \right) = \lim_{n \to \infty} P_\gamma \psi_n^+  - \lim_{n \to \infty} P_\gamma \psi_n^- \\
		= \lim_{n \to \infty} \left(\psi_{n+1}^+ - \psi_{n+1}^-\right) - f
		= u_\gamma -f.
	\end{align*}
	This shows ii) and therefore completes the proof.
\end{proof}

\subsection{Limit theorems}
This section contains our main results based on Assumption \ref{assumption_main} together with either, a bounded eccentricity condition, or, the existence of a Lyapunov function.

\subsubsection{Uniformly bounded eccentricity}\label{subsec:bounded_eccentricity}
The first two of our main results are for the setting where the eccentricities $E_\gamma$ are uniformly bounded.

\begin{ass}\label{ass:bounded_eccentricity}
	There exists $K \in (0, \infty) $ such that for any $\gamma \in \mathcal{I}$ and $y \in \cY$ we have 
	\[
	E_\gamma (y) \leq K.
	\]
\end{ass}

Our first main result is for the case $\beta\geq 1$ in the AIR scheme.

\begin{thm}\label{thm:limit_result_bounded_eccentricity_large_beta}
	Let $f \colon \cX \to \R$ be a function such that $h = h_f = f \cdot \mathds{1}_\Phi$ is Lipschitz with constant $L$ (w.r.t.~$d$) and bounded. 
	Let Assumptions~\ref{assumption_main} and \ref{ass:bounded_eccentricity} be true. 
	Then, for any $\beta\geq 1$ we have 
	\[
	\lim_{n \to \infty} \frac{1}{\sqrt{n} (\log n)^{1/2+\eps}} \left( \sum_{j=1}^n f(X_j) - \nu(f)\right) =0, 
	\]
	almost surely for arbitrary $\eps>0$. 
\end{thm}

Before we turn to the proof we state our result for $\beta \in (0,1)$, where the rate gets worse as $\beta $ approaches $0$.
This does not mean that small values of $\beta$ or purely adaptive algorithms are worse, rather this phenomena is due to our proof technique, which collapses as $\beta \to 0$. 
{However, we show in Section~\ref{subsec:small_beta_regime} that under additional assumptions this issue can be resolved.}

\begin{thm}\label{thm:limit_result_bounded_eccentricity_small_beta}
	Let $f \colon \cX \to \R$ be a function such that $h = h_f = f \cdot \mathds{1}_\Phi$ is Lipschitz with constant $L$ (w.r.t.~$d$) and bounded. 
	Let Assumptions~\ref{assumption_main} and \ref{ass:bounded_eccentricity} be true. 
	Then, for any $\beta \in (0,1)$ and $\eps > \frac{1}{1+\beta} - \frac{1}{2}$ we have 
	\[
	\lim_{n \to \infty} \frac{1}{n^{1/2+\eps}} \left( \sum_{j=1}^n f(X_j) - \nu(f)\right) =0, 
	\]
	almost surely. 
\end{thm}

To prove Theorems \ref{thm:limit_result_bounded_eccentricity_large_beta} and \ref{thm:limit_result_bounded_eccentricity_small_beta} we start with two auxiliary results.
\begin{lem}\label{lem:martingale_bound_for_bounded_eccentricity}
	{Let the assumptions of Theorem \ref{thm:limit_result_bounded_eccentricity_large_beta} or Theorem \ref{thm:limit_result_bounded_eccentricity_small_beta} be true. Then, $(M_n)_{n \in \N}$ is a $(\mathcal{F}_n)_{n \in \N}$ martingale and its differences satisfy}
	\[
	\vert \Delta_j \vert \le 2 \frac{k_0L M^{k_0} K}{1- \tau},
	\]
	almost surely, with $k_0, L, M , \tau$ and $K$ as in Assumptions \ref{assumption_main} and \ref{ass:bounded_eccentricity}. 
\end{lem}
\begin{proof}
	We observe that 
	\[
	\sup_{y \in \cY \,; \, \gamma \in \mathcal{I}} \vert u_\gamma (y) \vert \leq \frac{k_0L M^{k_0} K}{1- \tau},
	\]
	where we used Proposition \ref{prop:poisson_equation_wasserstein_setting} as well as Assumption \ref{ass:bounded_eccentricity}. 
	With the same arguments we deduce that 
	\[
	\sup_{y \in \cY \,; \, \gamma \in \mathcal{I}} \vert P_\gamma u_\gamma (y) \vert \leq \frac{k_0L M^{k_0} K}{1- \tau},
	\]
	and thus $\vert \Delta_j \vert \leq 2 \frac{k_0L M^{k_0} K}{1- \tau}$ almost surely.
	
	The bound for the differences implies $\E [\vert M_n\vert] = \E\left[\left\vert \sum_{j=1}^n \Delta_j \right\vert  \right] <\infty$ for any $n \in \N$.
	Hence, from Section \ref{sec:preliminaries} we know that $\E [M_n \vert \mathcal{F}_{n-1}] = M_{n-1}$ almost surely, showing that $M_n$ is a martingale. 
\end{proof}

Next we bound the remainder term $R_m$. 
Recall that the times of adaptation are $T_\ell = \sum_{j=1}^\ell k_j $, where $k_j$ are the ``windows'' between adaptation times $T_0 := 0$. 
Given $n \in \N$ we assume that $m = m(n)$ is such that $T_m \leq n+1< T_{m+1}$.
\begin{lem}\label{lem:R_m_bound_for_bounded_eccentricity}
	Let $h \colon \cX \to \R$ be a function such that $f = f_h = h \cdot \mathds{1}_\Phi$ is Lipschitz with constant $L$ (w.r.t.~$d$). 
	Let Assumptions~\ref{assumption_main} and \ref{ass:bounded_eccentricity} be true. 
	Then, for any $\beta>0$, we have
	\[
	\vert R_m \vert \leq  n^{1/(1+\beta)} \cdot (2c_\beta)^{1/(1+\beta)}\frac{k_0L M^{k_0} K}{1- \tau},
	\]
	almost surely. 
\end{lem}
\begin{proof}
	By Proposition \ref{prop:poisson_equation_wasserstein_setting} and  Assumption \ref{ass:bounded_eccentricity} we have
	\[
	\sup_{y \in \cY \,; \, \gamma \in \mathcal{I}} \vert u_\gamma (y) \vert \leq \frac{k_0L M^{k_0} K}{1- \tau},
	\]
	such that $\vert R_m \vert \le m \cdot  \frac{L M^{k_0} K}{1- \tau^{k_0}}$ almost surely. 
	According to \eqref{equ:AIR_MCMC_steps_bound} and since $T_m \le n+1$ it follows 
	\[
	m \le c_\beta^{1/(1+\beta)} T_m^{1/(1+\beta)} \le  c_\beta^{1/(1+\beta)} (n+1)^{1/(1+\beta)} \leq (2c_\beta)^{1/(1+\beta)} n^{1/(1+\beta)} ,
	\]
	which finishes the proof. 
\end{proof}

Now we are ready to show our main results of this section. 
\begin{proof}[Proof of Theorem \ref{thm:limit_result_bounded_eccentricity_large_beta}]
	Set $r(n ) = \sqrt{n}(\log n)^{1/2+\eps}$.
	By Lemma~\ref{lem:martingale_bound_for_bounded_eccentricity}, $(M_n)_{n \in \N}$ is a martingale with uniformly bounded increments. 
	Hence an application of Lemma~\ref{lem:martingale_almost_sure_convergence} shows 
	\[
	\lim_{n \to \infty}\frac{1}{r(n)} M_n =0
	\]
	almost surely.
	
	\medskip
	Since $\beta\geq 1$ and by Lemma \ref{lem:R_m_bound_for_bounded_eccentricity} we moreover deduce that 
	\[
	\frac{1}{r(n) } \vert R_m \vert \leq \frac{n^{1/(1+\beta)}}{r(n)} \widetilde{C} \leq \frac{n^{1/2}}{r(n)} \widetilde{C}, 
	\]
	with $\widetilde{C} = (2c_\beta)^{1/(1+\beta)}\frac{k_0L M^{k_0} K}{1- \tau}$. 
	Thus $\lim_{n \to \infty} \frac{1}{r(n) }\vert R_m \vert = 0$ a.s.
	
	Finally, by setting 
	\[
	g_n = h(Y_n) - h(Y_0) 
	+{u}_{\Gamma_0}(Y_0) - {u}_{\Gamma_{n}}(Y_n)
	\] 
	we have a sequence of uniformly (in $n$) bounded random variables. 
	Thus $\frac{1}{r(n)} g_n \to 0$ almost surely. 
	
	Now the result is a consequence of equations \eqref{equ:S_n_on_augmented_space} and  \eqref{equ:martingale_decomposition_air_MCMC_general}.
\end{proof}

\begin{proof}[Proof of Theorem \ref{thm:limit_result_bounded_eccentricity_small_beta}]
	We set $r(n) = n^{1/2+\eps}$.
	Note that our choice of $\eps$ implies that 
	\[
	\lim_{n \to \infty}\frac{n^{1/(1+\beta)}}{r(n)} =0.
	\]
	The proof is finished similar as in the proof of Theorem \ref{thm:limit_result_bounded_eccentricity_large_beta}. 
\end{proof}

\subsubsection{Lyapunov condition}
In this section we derive a limit result based on the existence of a Lyapunov function, which satisfies an inequality w.r.t. $d$.

Recall that a function $V \colon \cY \to [1, \infty)$ is called a Lyapunov function (for some transition kernel $P$ on $(\cY, \mathcal{F}_\cY)$) if there exist constants $\kappa \in (0,1)$ and $b \geq0$ such that 
\[
PV(y) \leq \kappa V(y)+b.
\]

Note that if $\pi$ is the invariant distribution of $P$, then $\pi(V) \leq \frac{b}{1-\tau}$, see e.g. \cite[Proposition 4.24]{hairer2006ergodic}.
\begin{ass}\label{ass:Lyapunov_existence}
	There exists a function $V \colon \cY \to [1, \infty)$ and constants $\kappa \in (0,1)$ and $b \in (0,\infty)$ such that for any $\gamma \in \mathcal{I}$ and $y \in \cY$,  
	\[
	P_\gamma V(y) \leq \kappa V(y) +b.
	\]
\end{ass}

If $V$ is a Lyapunov function then $V^{q}$ is also a Lyapunov function for any value of $q\in (0,1)$. 
This is well known from the literature (see e.g. \cite[Lemma 15.2.9]{meyn2012markov} for the case $q=1/2$).
For the sake of completeness we included a proof in the appendix, cf. Lemma \ref{lem:Lyapunov_function_exponent}. 

The second assumption relates the function $V^{1/p}$ to the distance-like function $d$, where $V$ is from Assumption \ref{ass:Lyapunov_existence} and $p>2$. 

\begin{ass}\label{ass:Lyapunov_bound}
	There are constants $p>2$ and $K \in (0, \infty)$  such that for the function $V$ from Assumption~\ref{ass:Lyapunov_existence} and any $y_1, y_2 \in \cY$ we have 
	\[
	d(y_1, y_2) \le K(V^{1/p}(y_1)+ V^{1/p}(y_2)).
	\]
\end{ass}

\begin{thm}\label{thm:limit_result_lyapunov}
	Let $f \colon \cX \to \R$ be a function such that $h = h_f = f \cdot \mathds{1}_\Phi$ is Lipschitz with constant $L$ (w.r.t. $d$) and 
	let Assumptions~\ref{assumption_main}, \ref{ass:Lyapunov_existence} and \ref{ass:Lyapunov_bound} be satisfied. 
	If $\vert h \vert \le V^{1/p}$, then for any $\beta \in (0,\infty)$ and $\eps > \max \left\{0,\frac{1}{1+\beta} + \frac{1}{p} - \frac{1}{2} \right\} $ we have 
	\[
	\lim_{n \to \infty} \frac{1}{n^{1/2+\eps}} \left( \sum_{j=1}^n f(X_j) - \nu(f)\right) =0, 
	\]
	almost surely. 
\end{thm}

We establish two auxiliary lemmas before proving the main result.

\begin{lem}\label{lem:martingale_bound_lyapunov_condition}
	Under the assumptions of Theorem \ref{thm:limit_result_lyapunov} $(M_n)_{n \in \N}$ is a $(\mathcal{F}_n)_{n \in \N}$ martingale and there exists some constant $C_1 \in (0,\infty)$ such that
	\[
	\E	\vert \Delta_j \vert^2 \le C_1,
	\]
	for any $j \in \N$.
\end{lem}

\begin{proof}
	First, we establish the claimed bound on the differences. 
	
	For arbitrary $j,k \in \N$ Proposition~\ref{prop:poisson_equation_wasserstein_setting}, our assumptions about the Lyapunov functions $V$ and $V^{1/p}$ as well as the bound $\pi_\gamma (V) \leq \frac{b}{1-\kappa}$, valid for any $\gamma \in \mathcal{I}$ (compare \cite[Proposition 4.24]{hairer2006ergodic}), yield
	\begin{align*}
		\E [\vert u_{\Gamma_{j}} (Y_k)\vert^p ] 
		&\leq 
		\left( \frac{k_0L M^{k_0} }{1- \tau}\right)^p  \E [ E_{\Gamma_{j}} (Y_k)^p] \\
		&\leq 
		\left( \frac{k_0L M^{k_0} }{1- \tau}\right)^p  \E \left[ \int_\cX d(Y_k, y)^p \pi_{\Gamma_j} (\d y)   \right] \\
		&\leq 
		\left( \frac{Kk_0L M^{k_0} }{1- \tau}\right)^p  \E \left[ \int_\cX \left(V^{1/p}(Y_k) + V^{1/p}(y) \right)^{p} \pi_{\Gamma_j} (\d y)   \right] \\
		&\leq 
		\left( \frac{Kk_0L M^{k_0} }{1- \tau}\right)^p  \E \left[ \int_\cX2^p \left(  V(Y_k) + V(y)\right)  \pi_{\Gamma_j} (\d y)   \right] \\
		&\leq 
		\left( \frac{2Kk_0L M^{k_0} }{1- \tau}\right)^p  \left( \E \left[ V(Y_k)   \right] + \frac{b}{1-\kappa} \right) .
	\end{align*}
	By the Lyapunov property $\E[V(Y_k)] = \E [P_{\Gamma_{k-1}}V(Y_{k-1})] \leq \kappa \E[V(Y_{k-1})] +b$.
		Hence, by induction on $k$, we obtain $\E[V(Y_k)] \leq \kappa^k \E[V(Y_{0})] +b\sum_{\ell=0}^{k-1}\kappa^\ell$ and therefore
		\begin{align*}
			\E \left[ V(Y_k)   \right] &\leq \kappa^k \E[V(Y_{0})] +b\sum_{\ell=0}^{k-1}\kappa^\ell \leq \kappa^k \E[V(Y_{0})] +b\sum_{\ell=0}^{\infty}\kappa^\ell \\ 
			&= 
			\kappa^k \E[V(Y_0)] + \frac{b}{1-\kappa} \leq \E[ V(Y_0)] + \frac{b}{1-\kappa}
			= V(y_0) + \frac{b}{1-\kappa}.
	\end{align*}
	Applying Jensen's inequality yields 
	\begin{align*}
		\E [\vert u_{\Gamma_{j}} (Y_k)\vert^2 ] 
		\leq
		\E [\vert u_{\Gamma_{j}} (Y_k)\vert^p ]^{2/p} 
		\leq   
		\left( \frac{2Kk_0L M^{k_0} }{1- \tau}\right)^2  \left( V(y_0) + \frac{2b}{1-\kappa} \right)^{2/p} < \infty.
	\end{align*}
	Employing Proposition \ref{prop:poisson_equation_wasserstein_setting} ii) we deduce 
	\[
	\vert P_{\Gamma_{j}} u_{\Gamma_{j}}(Y_k) \vert 
	\le   
	\vert u_{\Gamma_{j}}(Y_k) \vert + \vert V^{1/p}(Y_k) \vert + \pi_{\Gamma_j}(\vert V^{1/p} \vert ).
	\]
	Using the same arguments as above, one obtains a bound for $\vert P_{\Gamma_{j}} u_{\Gamma_{j}}(Y_k) \vert^p$ from which the desired estimate for $\E\vert \Delta_j \vert^2$ follows. 
	
	Using this bound on the differences it follows that $\E [\vert M_n \vert] < \infty$ for any $n \in \N$. Therefore, from Section \ref{sec:preliminaries} we know that $\E [M_n \vert \mathcal{F}_{n-1}] = M_{n-1}$ almost surely, such that $(M_n)_{n \in \N}$ is a martingale as claimed. 
\end{proof}

Next we bound the remainder term $R_m$ in a suitable way.

\begin{lem}\label{lem:remainder_bound_lyapunov_condition}
	Let the assumptions of Theorem \ref{thm:limit_result_lyapunov} be satisfied. Then, there exists some $C_2 \in (0, \infty)$ such that for any $n \in \N$ and $m=m(n)$ we have
	\[
	\E	\vert R_m \vert^p \le m^p \cdot  C_2.
	\]
\end{lem}

\begin{proof}
	From the proof of Lemma \ref{lem:martingale_bound_lyapunov_condition} we know 
	\[
	\E [\vert u_{\Gamma_{j}} (Y_k)\vert^p ] 
	\le 
	\left( \frac{2Kk_0L M^{k_0} }{1- \tau}\right)^p  \left( V(y_0) + \frac{2b}{1-\kappa}  \right),
	\]
	for any $j,k \in \N$.
	Using this estimate and Minkowski's inequality the bound for $R_m$ follows.
\end{proof}

Now we turn to the proof of Theorem \ref{thm:limit_result_lyapunov}.

\begin{proof}[Proof of Theorem \ref{thm:limit_result_lyapunov}]
	We show that all terms in the martingale decomposition \eqref{equ:martingale_decomposition_air_MCMC_general} go to $0$ after being normalised with $r(n) = n^{1/2+\eps}$.
	
	We begin with the martingale part $M_n$ which has differences uniformly bounded in $L^2$, see Lemma~\ref{lem:martingale_bound_lyapunov_condition}.
	Thus Lemma \ref{lem:martingale_almost_sure_convergence} yields 
	\[
	\lim_{n \to \infty}
	\frac{1}{r(n)} M_n =0
	\]
	almost surely.

Next we turn to $R_m$: For any $\delta>0$ we have by Markov's inequality, Lemma~\ref{lem:remainder_bound_lyapunov_condition} and \eqref{equ:AIR_MCMC_steps_bound} that
	\[
	\mathbb{P}\left[ \left\vert R_m \right\vert > \delta r(n) \right] \leq 
	\frac{m^p}{r(n)^p} C_2 \delta^{-p} 
	\leq
	\frac{n^{p/(1+\beta)}}{r(n)^p} C_2 (2c_\beta)^{1/(1+\beta)} \delta^{-p}.
	\]
	Our assumption on $\eps$ implies $p(\frac{1}{1+\beta}-\frac{1}{2} - \eps) < -1$, whence
	\[
	\sum_{n=1}^\infty\frac{n^{p/(1+\beta)}}{r(n)^p} = \sum_{n=1}^\infty\frac{n^{p/(1+\beta)}}{n^{p(1/2+\eps)}} <\infty,
	\]
	such that $\lim_{n \to \infty}\frac{1}{r(n)}R_m =0$ almost surely by the Borel-Cantelli lemma.
	
	\medskip
	The remaining part that needs to be controlled is:
	\[
	\frac{1}{r(n)}\left( h(Y_n) - h(Y_0) 
	+{u}_{\Gamma_0}(Y_0) - {u}_{\Gamma_{n}}(Y_n) \right) .
	\]
	Note that by assumption $Y_0 = y_0 \in \cY$ and $\Gamma_0=\gamma_0 \in \mathcal{I}$ are fixed, which immediately implies that
		\[
		\lim_{n \to \infty}	\frac{1}{r(n)}h(Y_0) =0 \qquad \text{and} \qquad \lim_{n \to \infty}\frac{1}{r(n)}u_{\Gamma_0}(Y_0) = 0.
		\]
		Moreover, we have shown in the proof of Lemma \ref{lem:martingale_bound_lyapunov_condition} that $\E \vert u_{\Gamma_{j}} (Y_k) \vert^p \leq C $, for any $j,k \in \N$; with some constant $C \in (0, \infty)$ independent of $k,j$.
		Hence, Markov's inequality and the assumptions on $\eps$ and $p$ yield
		\begin{equation}\label{equ:Borel_cantelli_remainder_part}
			\sum_{n=1}^\infty \mathbb{P}\left[ \vert u_{\Gamma_n}(Y_n) \vert \geq \delta r(n)  \right] \leq \sum_{n=1}^{\infty} \frac{\E \vert u_{\Gamma_{n}} (Y_n) \vert^p }{\delta^p r(n)^p} \leq \sum_{n=1}^{\infty} \frac{C^p}{\delta^p n^{p(1/2+\eps)}} < \infty.
		\end{equation}
		Thus $\lim_{n \to \infty} \frac{1}{r(n)} u_{\Gamma_n}(Y_n) =0$ almost surely, by the Borel-Cantelli lemma.
		
		Finally, since $\vert h \vert \leq V^{1/p}$ we have $\E [\vert h(Y_n) \vert^p ] \leq \E[V(Y_n)]$. Moreover, we have shown in the proof of Lemma~\ref{lem:martingale_bound_lyapunov_condition} that $\E[V(Y_n)] \leq V(y_0) + \frac{b}{1-\kappa} < \infty$.
		Thus, the same steps as in \eqref{equ:Borel_cantelli_remainder_part} show that $\sum_{n=1}^\infty \mathbb{P}\left[ \vert h(Y_n) \vert \geq \delta r(n)  \right] < \infty$, such that also $\lim_{n \to \infty}\frac{1}{r(n)}h(Y_n) =0$ almost surely. 
\end{proof}

	\subsection{The ``small $\beta$ regime''}\label{subsec:small_beta_regime}
We refer to the ``small $\beta$ regime'' in the case of $\beta\in(0,1)$, i.e. exactly when the rates in the Theorems \ref{thm:limit_result_bounded_eccentricity_large_beta}, \ref{thm:limit_result_bounded_eccentricity_small_beta} and \ref{thm:limit_result_lyapunov} deteriorate.
	The main reason for this are the bounds for $R_m$, see Lemmas \ref{lem:R_m_bound_for_bounded_eccentricity} and \ref{lem:remainder_bound_lyapunov_condition}, which later lead to estimates involving the ratio $m/n$. 
	Below we show that under additional assumptions, better rates for $\beta\in (0,1)$ are possible. 
	For the sake of presentation we focus on the setting of Section \ref{subsec:bounded_eccentricity}. 
	The following Lipschitz and waning adaptation\footnote{{Waning adaptation was recently introduced in \cite{laitinen2024invitationadaptivemarkovchain} and generalises diminishing adaptation, which is a commonly used assumption in the adaptive MCMC literature.}} conditions are the main ingredient to improve the bound for $R_m$. 
	\begin{ass}\label{ass:wanig_adaptation_and_Lipschitz}
		For the parameter space we have $\mathcal{I}\subseteq \R^s$, with $s \in \N$, and the function $u_\gamma$ of Proposition \ref{prop:poisson_equation_wasserstein_setting} satisfies a uniform Lipschitz condition,
		\begin{equation}\label{equ:Lipschitz_parameter_Poisson_solution}
			\sup_{y \in \cY}\vert u_\gamma(y) - u_{\gamma'}(y) \vert \leq \Lambda \norm{\gamma- \gamma'},
		\end{equation}
		with $\Lambda \in (0, \infty)$ and $\norm{\cdot}$ being the Euclidean norm on $\R^s$. 
		Moreover, the sequence of parameters is (strongly) $\rho$-waning for some $\rho \in (0,1)$, that is, 
		\begin{equation}\label{equ:waning_adaptation}
			\lim_{k\to \infty} \frac{1}{k^\rho} \sum_{j=1}^k \norm{\Gamma_{T_j}-\Gamma_{T_{j-1}}} =0, \qquad \text{almost surely}.
		\end{equation}
	\end{ass}
	
	\begin{rem}
		The Lipschitz condition \eqref{equ:Lipschitz_parameter_Poisson_solution} is somewhat restrictive, since it is uniformly in $y \in \cY$. 
		However, it fits into the setting of simultaneous uniform ergodicity (studied in Section \ref{sec:examples}), see e.g. Lemma 8 and Section 5 in \cite{laitinen2024invitationadaptivemarkovchain}; see also \cite[Proposition 3]{andrieu2006ergodicity} and \cite[Proposition 4]{saksman2010ergodicity}.
	\end{rem}
	
	Using Assumption \ref{ass:wanig_adaptation_and_Lipschitz} we are able to improve Theorem \ref{thm:limit_result_bounded_eccentricity_small_beta}.
	\begin{thm}\label{thm:small_beta_waning_adaptation}
		Suppose we are in the setting of Theorem \ref{thm:limit_result_bounded_eccentricity_small_beta}, 
		and additionally let Assumption \ref{ass:wanig_adaptation_and_Lipschitz} be satisfied.
		Then, for any $\beta \in (0,1)$ with $\beta \geq 2\rho -1$ and any $\eps>0$ we have 
		\[
		\lim_{n \to \infty} \frac{1}{\sqrt{n}\log(n)^{1/2+\eps}} \left( \sum_{j=1}^n f(X_j) - \nu(f)\right) =0, 
		\]
		almost surely. 
	\end{thm}
	
	\begin{proof}
		Note that all assumptions to apply Lemma \ref{lem:martingale_bound_for_bounded_eccentricity} are fulfilled. 
		Hence, as in the proof of Theorem \ref{thm:limit_result_bounded_eccentricity_large_beta} we may apply Lemma \ref{lem:martingale_almost_sure_convergence} to obtain 
		\[
		\lim_{n \to \infty}
		\frac{1}{\sqrt{n}(\log n)^{1/2+\eps}} M_n =0,
		\]
		almost surely. Moreover, by definition of $R_m$ and \eqref{equ:Lipschitz_parameter_Poisson_solution} we have 
		\[
		\left\vert \frac{1}{\sqrt{n}(\log n)^{1/2+\eps}} R_m \right\vert \leq  \frac{\Lambda m^\rho}{\sqrt{n}(\log n)^{1/2+\eps}} \frac{1}{m^\rho}\sum_{j=1}^m \norm{\Gamma_{T_j}-\Gamma_{T_{j-1}}}  .
		\]
		Bounding $m=m(n)$ as in the proof of Lemma \ref{lem:R_m_bound_for_bounded_eccentricity}, and using the assumption on $\beta$ implies $\frac{m^\rho}{\sqrt{n}} \leq \frac{c n^{\rho/(1+\beta)}}{\sqrt{n}} \leq c$ for any $n \in \N$, where $c= (2c_\beta)^{\rho/(1+\beta)}$. 
		Combining this with the waning adaptation condition \eqref{equ:waning_adaptation} it follows immediately that $\lim_{n \to \infty}\frac{1}{\sqrt{n}(\log n)^{1/2+\eps}} R_m =0$ almost surely. 
		
		Finally, the remaining part of the martingale decomposition $h(Y_n) - h(Y_0) 
		+{u}_{\Gamma_0}(Y_0) - {u}_{\Gamma_{n}}(Y_n)$ is treated as in the proof of Theorem \ref{thm:limit_result_bounded_eccentricity_large_beta}.
	\end{proof}
	\subsection{Discussion, open questions and outlook}
	We have proven limit theorems for AIR algorithms based on a contraction condition, see Assumption \ref{assumption_main}. 
	If $\eps >0$ can be chosen close to $0$ in Theorems~\ref{thm:limit_result_bounded_eccentricity_large_beta}, \ref{thm:limit_result_bounded_eccentricity_small_beta}, and \ref{thm:limit_result_lyapunov}, then the rates in these results almost match the one in the law of the iterated logarithm.
	The choice of $\eps$ depends on $\beta$, which determines how often we adapt and larger values of $\beta$ allow for smaller $\eps$ in our results. 
	Yet, under additional assumptions, cf. Assumption \ref{ass:wanig_adaptation_and_Lipschitz}, also the ``small $\beta$ regime'' can be separately handled, see Theorem \ref{thm:small_beta_waning_adaptation}. 
	It is fair to ask whether additional conditions are really required to obtain ``good rates'' for small $\beta$. 
	The treatment of this demanding question is left open.

	Moreover, convergence rates as in Theorems \ref{thm:limit_result_bounded_eccentricity_large_beta} or Theorem \ref{thm:limit_result_lyapunov} are also desirable for purely adaptive, i.e. $\beta=0$, algorithms.
	Such may be deduced under approriate assumptions by virtue of Proposition~6 of \cite{saksman2010ergodicity}.
	The authors of \cite{saksman2010ergodicity} study \textit{stochastic approximation} type adaptive MCMC methods, and to apply their Proposition 6, a constant $\eps>0$ (different from our $\eps$) appearing in their Assumptions (A1) -- (A4) needs to be sufficiently small\footnote{{For the adaptive random walk Metropolis algorithm this is possible (under certain assumptions), see \cite{saksman2010ergodicity}, in particular the proof of Theorem 10 and Remark 16 there.}}.
	Their assumptions are close to our geometric ergodicity example of Section \ref{sec:examples}, where $\mathcal{W}$ is based upon a weighted trivial metric. 
	It would be interesting  to extend the results of \cite{saksman2010ergodicity} to the Wasserstein-like function setting of the present paper. 
	
	{Weakening Assumption \ref{assumption_main} in the context of stochastic approximation type algorithms (like those in \cite{andrieu2006ergodicity,saksman2010ergodicity}) may be particularly rewarding.}
	Namely, these methods naturally possess certain stabilisation properties concerning the adaptation, which may be advantageously exploited in proofs.
	Potentially this allows	to drop the uniformity of $\gamma$ in Assumption \ref{assumption_main}, and can	also improve the rates for the ``small $\beta$ regime''.

\section{Examples}\label{sec:examples}

In this section we apply our theoretical results to three different settings, depending on the ergodicity properties of the family $\{ P_\gamma\}_{\gamma \in \mathcal{I}}$.
We start with the case where all kernels are uniformly ergodic,
then, we investigate an AIR scheme relying on geometric ergodic kernels.
Finally, we study algorithms where all the kernels satisfy conditions similar to those shown in \cite{hairer2011asymptotic}. 

Recall that there is some metric $d_\cY$ on $\cY$ such that the space is Polish. 
	If we speak of (semi-) continuity we always refer to that metric. 
\subsection{Uniform Ergodicity}
Here we set $d(y_1,y_2) = \mathds{1}_{\{y_1 \neq y_2\}}$, i.e. $d$ is the trivial metric.
In this case $\mathcal{W}$ coincides with the total variation distance, that is,
\[
\mathcal{W}(\mu_1, \mu_2) = \norm{\mu_1 - \mu_2}_{tv},
\]
for any two probability measures $\mu_1, \mu_2$ on $(\cY, \mathcal{F}_\cY)$.

We rely on the following
\textit{simultaneous uniform ergodicity} condition.

\begin{ass}\label{ass:uniform_ergodicity}
	There exist $\tau \in [0,1)$ and $\Lambda \in (0, \infty)$ such that for any $\gamma \in \mathcal{I}$, $y \in \cY$ and $\ell \in \N$ we have 
	\[
	\norm{P_\gamma^\ell (y, \cdot) - \pi_\gamma}_{tv} \leq \Lambda \tau^\ell. 
	\]
\end{ass} 

\begin{rem}
	Assumption \ref{ass:uniform_ergodicity} is common in the literature, see e.g. \cite{roberts2007coupling,laitinen2024invitationadaptivemarkovchain}, and the references therein. If $\cY\subseteq \R^d$ is compact, then Assumption \ref{ass:uniform_ergodicity} is true for certain adaptive Metropolis algorithms, see e.g. \cite[Section 5]{laitinen2024invitationadaptivemarkovchain}.
\end{rem}

We add another example which allows for an unbounded state space.
	\begin{example}\label{ex:Metropolis_geometric_ergodic}
		We consider adaptive stereographic MCMC algorithms as recently introduced in \cite{bell2025adaptivestereographicmcmc}. 
		In this setting $\cX  = \R^d$ and $\pi_\gamma \equiv \nu$, such that the auxiliary space $\Phi$ is not needed and can be chosen to be a singleton. 
		The main idea is to use the preconditioned stereographic projection to map $\nu$ to the sphere, where either a random walk Metropolis or a geodesic slice sampler\footnote{{See \cite{habeck2024geodesicslicesamplingsphere} for more details about geodesic slice sampling on the sphere.}} is employed.  
		Roughly speaking, adaptation is done by updating the projection map, and, if necessary, some further tuning parameters, we refer to \cite{bell2025adaptivestereographicmcmc}. 
		Suppose the density of $\nu$, denoted by $p\colon \R^d \to [0, \infty)$, satisfies 
		\[
		\limsup_{\norm{x} \to \infty}\left(p(x) (\norm{x}+1)^d \right) < \infty, 
		\]
		where $\norm{\cdot}$ is the Euclidean norm. Then, Assumption \ref{ass:uniform_ergodicity} is true for the adaptive stereographic algorithms of \cite{bell2025adaptivestereographicmcmc}, cf. Lemma 4.1 and Lemma 4.2 there.  
	\end{example}

\begin{thm}\label{thm:convergence_uniform_ergodic_case}
	Let $f \colon \cX \to \R$ be bounded and Assumption \ref{ass:uniform_ergodicity} be true. 
	If $\beta\geq 1$, then for any $\eps>0$ we have that almost surely 
	\[
	\lim_{n \to \infty}\frac{1}{\sqrt{n}(\log n )^{1/2+\eps}} \sum_{j=1}^n \left( f(X_j) - \nu(f)\right) =0.
	\]
\end{thm}

\begin{proof}
	Assumption \ref{assumption_main} is satisfied since the total variation distance and $\mathcal{W}$ coincide under the assumptions of this section. 
	Moreover, by definition of $E_\gamma(y)$ we have
	\[
	E_\gamma(y) = \int_\cY d(y, y') \pi_\gamma (\d y') \leq \int_\cY 1 \; \pi_\gamma(\d y') = 1,
	\]
	for any $\gamma \in \mathcal{I}$. 
	Hence also Assumption \ref{ass:bounded_eccentricity} is true. 
	Since $f$ is bounded it is also Lipschitz w.r.t. $d$, so all requirements to apply Theorem \ref{thm:limit_result_bounded_eccentricity_large_beta} are fulfilled. 
\end{proof}

\begin{cor}\label{cor:convergence_speed_uniform_ergodicity}
	Let $f \colon \cX \to \R$ be bounded, Assumption \ref{ass:uniform_ergodicity} be true, $\beta \geq1$ and $\eps >0$.
	Then, there exists some (measurable) $C \colon \Omega \to \R$ such that for $\mathbb{P}$-almost all $\omega \in \Omega$ and any $n \in \N$ we have the following bound,
	\[
	\left\vert  \frac{1}{{n}} \sum_{j=1}^n \left(  f(X_j(\omega)) - \nu(f)  \right)  \right\vert \leq \frac{(\log n)^{1/2+\eps}}{\sqrt{n}}C(\omega) .
	\]
\end{cor}

	\begin{rem}
		A result similar to the one of Theorem~\ref{thm:convergence_uniform_ergodic_case}
		for $\beta \in (0,1)$, with worse rate, can be derived (under the same assumptions) by employing Theorem \ref{thm:limit_result_bounded_eccentricity_small_beta} (instead of Theorem~\ref{thm:limit_result_bounded_eccentricity_large_beta}).  
		However, if additionally Assumption \ref{ass:wanig_adaptation_and_Lipschitz} is satisfied, then Theorem \ref{thm:small_beta_waning_adaptation} can be used to extend Theorem \ref{thm:convergence_uniform_ergodic_case} to the case $\beta\in(0,1)$ (with the same rate as for $\beta\geq 1$).
	\end{rem}

\subsection{Geometric ergodicity}\label{subsec:Geometric_Ergodicity}
Next we consider the case where we have a (simultaneous) drift and minorisation condition.
Assumptions of this form are standard in adaptive MCMC theory, see e.g. \cite{andrieu2006ergodicity,roberts2007coupling,fort2011convergence,saksman2010ergodicity,chimisov2018air}.

\begin{ass}\label{ass:drift_and_minorisation_new}
	There exists $C \in \mathcal{F}_\cY$, a lower semi-continuous measurable function $V \colon \cY \to [1,\infty)$ as well as numbers $\kappa \in (0,1)$, $\delta >0$, $b < \infty$ such that $\sup_{x\in C} V(x) = v< \infty$ and for any $\gamma \in \mathcal{I}$ the following properties are satisfied:
	\begin{itemize}
		\item[i)] For any $y \in \cY$ we have		$P_\gamma V(y) \leq \kappa V(y) + b\mathds{1}_C(y) $. 
		\item[ii)] 	There exists a probability measure $\eta_\gamma$ on $C$ such that for any ${x \in C}$ we have $P_\gamma (x, A )\geq \delta \eta_\gamma(A\cap C)$ for all $A\in\mathcal{F}_\cY$.
	\end{itemize}
\end{ass}

	\begin{example}
	For $d \in \N$ let $\cX = \R^d$ and assume that probability measure $\nu$ satisfies the requirements of \cite[Theorem 10]{saksman2010ergodicity} (super-exponential tails and regular contours), see also \cite{JARNER2000geometric}.
		Moreover, let $\mathcal{I} \subseteq \R^{d \times d}$ be the set of positive definite matrices with all eigenvalues in the interval $[a_1,a_2]$, for some $0< a_1 < a_2 <\infty$.
		Assume that $\pi_\gamma \equiv \nu$ for any $\gamma \in \mathcal{I}$, and that $\Phi$ is a singleton, such that we can identify $\cY$ with $\cX= \R^d$. 
		
		For each $\gamma \in \mathcal{I}$ let $P_\gamma$ correspond to a random walk Metropolis algorithm, using a Gaussian proposal with mean $0\in \R^d$ and covariance $\gamma$. 
		Furthermore, let $V(x)= \sqrt{\sup_{z\in\cX} p(z)}p(x)^{-1/2}$, with $p$ being the density of $\nu$. 
		Then, by virtue of \cite[Proposition 15]{saksman2010ergodicity} it follows that Assumption \ref{ass:drift_and_minorisation_new} is satisfied.  
		In particular, AIR Metropolis algorithms based on covariance approximation of $\nu$ fit into the setting of this example. 
		It should be mentioned that in this context \cite[Proposition 6]{saksman2010ergodicity} can be used to derive convergence rates, similar to those below, for the classical adaptive random walk Metropolis algorithm, where the covariance of $\nu$ is estimated via the empirical covariance of $(X_0, \dots, X_n)$.
	\end{example}

With $V$ from Assumption \ref{ass:drift_and_minorisation_new} let us introduce the following family of lower semi-continuous metrics
\[
d_q(y_1,y_2) = \mathds{1}_{\{y_1 \neq y_2\}}\left( V^q(y_1) + V^q(y_2)\right),
\]
where $q \in (0,1]$. 
Satisfying all properties of a metric, $d_q$ is clearly also a distance-like function.
The set Lipschitz functions with constant $1$ w.r.t. $d_q$ is denoted by $\text{Lip}_{q,1}$, and the corresponding $1$-Wasserstein distance by $\mathcal{W}_q$. 

\begin{thm}\label{thm:dirft_minorisation_main_new}
	Let Assumption \ref{ass:drift_and_minorisation_new} be true and $f \colon \cX \to \R$ such that $h = h_f = \mathds{1}_\Phi f$ satisfies $\vert h \vert \leq V^{1/p} $, with $V$ from Assumption~\ref{ass:drift_and_minorisation_new} and some $p>2$.
	Then, for $\beta>0$ and $\eps> \max\left\{0, \frac{1}{1+\beta} + \frac{1}{p} - \frac{1}{2}\right\}$  we have almost surely
	\[
	\lim_{n \to \infty} \frac{1}{n^{1/2+\eps}} \sum_{j=1}^n \left( f(X_i) - \nu(f) \right) =0.
	\]
\end{thm}

\begin{proof}
	Set $q=1/p \in (0,1]$.
	The goal is to apply Theorem \ref{thm:limit_result_lyapunov}.  
	In the same fashion as \cite[Lemma 2.1]{hairer2011yet} one can prove 
	\[
	\mathcal{W}_q(\mu_1, \mu_2) 
	=
	\sup_{\vert \psi \vert \leq V^{q}} \left\vert  \int_\cY \psi \; \d \mu_1  - \int_\cY \psi \; \d \mu_2 \right\vert,
	\]
	where $V $ is the function from Assumption~\ref{ass:drift_and_minorisation_new} and $\mu_1, \mu_2$ are arbitrary probability measures on $(\cY, \mathcal{F}_\cY)$. 
	Note that the conditions formulated in Assumption~\ref{ass:drift_and_minorisation_new} do also hold for $V^q$ instead of $V$. Therefore, 
	Assumption \ref{ass:drift_and_minorisation_new} implies simultaneous geometric ergodicity:
	There exist $\alpha_q \in \R$ and $\tau_q \in (0,1)$ only depending on $\delta, \kappa, b, v,q$ such that for any $\gamma \in \mathcal{I}$, $y \in \cY$ we have 
	\[
	{\mathcal{W}_q(P_\gamma^\ell(y,\cdot ), \pi_\gamma) = } 
	\sup_{\vert \psi \vert \leq V^{q}} \left\vert  \int_\cY \psi(y') P_\gamma^\ell (y, \d y')  - \pi_\gamma(\psi)\right\vert \leq \alpha_q V^q(y) \tau_q^\ell,
	\]
	for any $\ell \in \N$.
	This follows from \cite{Baxendale2005renewal}, see also \cite[Theorem 2]{andrieu2006ergodicity} \cite[Section 7]{roberts2007coupling}.
	
	Hence, for any $y_1 , y_2 \in \cY$ and any $\ell \in \N$ an application of the triangle inequality yields
	\begin{align*}
		\mathcal{W}_q(P^\ell_\gamma(y_1, \cdot ), P^\ell_\gamma(y_2, \cdot )  )
		\leq \mathcal{W}_q(P_\gamma^\ell(y_1,\cdot ), \pi_\gamma) +  \mathcal{W}_q(P_\gamma^\ell(y_2,\cdot ), \pi_\gamma) \\
		\leq \alpha_q \tau_q^\ell (V^q(y_1) + V^q(y_2)).
	\end{align*}
	This shows that Assumption~\ref{assumption_main} is true (for $\mathcal{W}_q$).
	Moreover, Assumptions~\ref{ass:Lyapunov_existence} and \ref{ass:Lyapunov_bound} (both w.r.t. $d_q$) are immediate consequences of Assumption \ref{ass:drift_and_minorisation_new}.
	
	Finally, we note that since $\vert h \vert \leq V^q$, it follows that $h$ is Lipschitz (with constant $1$) w.r.t. $d_q$. 
	Thus all requirements to apply Theorem \ref{thm:limit_result_lyapunov} are met and the claimed result follows. 
\end{proof}

\begin{cor}\label{cor:rate_geometric_ergodic}
	Let the assumptions of Theorem \ref{thm:limit_result_lyapunov} be satisfied. Then, for any $\beta \in (0,\infty)$ and $\eps>\max\left\{0, \frac{1}{1+\beta}+\frac{1}{p}- \frac{1}{2} \right\}$ there exists some (measurable) $C \colon \Omega \to \R$ such that for $\mathbb{P}$-almost all $\omega \in \Omega$ and any $n\in \N$ we have 
	\[
	\left\vert  \frac{1}{{n}} \sum_{j=1}^n \left(  f(X_j(\omega)) - \nu(f)  \right)  \right\vert \leq \frac{1}{n^{1/2 - \eps}}C(\omega) .
	\]
\end{cor}

Clearly, we want to have $\eps>0$ as small as possible in Corollary~\ref{cor:rate_geometric_ergodic} to obtain 
an as good as possible rate of convergence. Intuitively, for fixed $\beta $ we may choose $\eps$ small if $p$ is large. 
This can be interpreted as follows: For large $p$ the requirement $\vert h \vert \leq V^{1/p}$ gets more restrictive, such that ``easier'' functions are considered, leading to a better rate.  

\subsection{Weak Harris ergodicity}
In this section we investigate kernels which satisfy a contraction condition similar to the one which one obtains from the \textit{weak Harris theorem} of \cite{hairer2011asymptotic}.
For certain (classical) MCMC algorithms contraction results, similar to those required in this section, were shown in \cite{hairer2014spectralgaps,hosseini2023spectral}.
In the adaptive setting, however, we are not aware of convergence results which are based on this kind of contraction assumptions.  

Throughout the section we assume that $d \colon \cY \times \cY \to [0,1]$ is a distance-like function. 
\begin{ass}\label{ass:weak_harris_ergodicity_Lyapunov}
	There exists a lower semi-continuous function $V \colon \cY \to [1, \infty)$ which is a Lyapunov function for all $P_\gamma$ simultaneously, i.e. Assumption \ref{ass:Lyapunov_existence} is satisfied.
\end{ass}

Following \cite{hairer2011asymptotic} we define the following distance-like function(s) 
\[
\widetilde{d}_q (y_1, y_2) = \sqrt{d(y_1, y_2) (1+ V^q(y_1) + V^q(y_2))},
\]
where $q \in (0,1)$ and $V$ is from Assumption \ref{ass:weak_harris_ergodicity_Lyapunov}.
We denote the corresponding Wasserstein-like function and contraction coefficient by $\widetilde{\mathcal{W}}_q$ and $\widetilde{\tau}_q$, respectively.
Moreover, we rely on the following requirement.
\begin{ass}\label{ass:weak_harris_ergodicity_contraction}
	There exist numbers $k_0 \in \N$, $\tau \in [0,1)$, $M \in (0, \infty)$ and $p > 2$ such that for any $\gamma \in \mathcal{I}$ we have
	\[ 
	\widetilde{\tau}_{1/p}(P_\gamma^{k_0})
	\leq \tau 
	\qquad
	\text{ and } \qquad
	\widetilde{\tau}_{1/p}(P_\gamma)
	\leq M.
	\]
\end{ass}

\begin{rem}
	A contraction assumption w.r.t. $\widetilde{\mathcal{W}}_q$ can be checked by using the weak Harris theorem of \cite{hairer2011asymptotic}.
	It is worth noting that the mentioned result is suitable for infinite dimensional settings, see also \cite{hairer2014spectralgaps,hosseini2023spectral}. 
\end{rem}

\begin{thm}\label{thm:weak_harris_ergodic_limit_theorem}
	Let Assumptions \ref{ass:weak_harris_ergodicity_Lyapunov} and \ref{ass:weak_harris_ergodicity_contraction} be true and $f \colon \cX \to \R$ be such that $h = f \cdot \mathds{1}_{\Phi}$ is Lipschitz w.r.t. $\widetilde{d}_{1/p}$  and $\vert h \vert\le V^{1/p}$, where $V,p $ are as in Assumption \ref{ass:weak_harris_ergodicity_contraction}.
	Then, for any $\beta>0$ and $\eps> \max\left\{0, \frac{1}{1+\beta} + \frac{1}{p} - \frac{1}{2}\right\}$ we have that almost surely
	\[
	\lim_{n \to \infty} \frac{1}{n^{1/2+\eps}} \sum_{j=1}^n \left( f(X_i) - \nu(f) \right) =0.
	\]
\end{thm}
\begin{proof}
	Clearly, Assumptions \ref{assumption_main} and \ref{ass:Lyapunov_existence} are true.
	Moreover, by definition of $\widetilde{d}_{1/p}$ also Assumption \ref{ass:Lyapunov_bound} (w.r.t. $\widetilde{d}_{1/p}$) is satisfied. 
	
	Since $h$ is Lipschitz and bounded by $V^{1/p}$ all requirements to apply Theorem \ref{thm:limit_result_lyapunov} are met and the result follows. 
\end{proof}

A result about path-wise convergence, akin to Corollary \ref{cor:rate_geometric_ergodic}, can be deduced from Theorem \ref{thm:weak_harris_ergodic_limit_theorem}. However, we omit a detailed discussion here.

\begin{rem}
	A contraction for the pCN kernel w.r.t. $\widetilde{\mathcal{W}}_q$ was shown in \cite{hairer2014spectralgaps}, being valid for a tuning parameter between $(0,\frac{1}{2}]$.
	Hence Theorem \ref{thm:weak_harris_ergodic_limit_theorem} may be useful to analyse the convergence of an AIR version of that method. 
	Ideas about the adaptation of a pCN algorithm can be found in \cite{LAW2014127,HU2017492,rudolf2018generalization}.
\end{rem}

\section*{Acknowledgements}
The authors thank Mareike Hasenpflug and Thomas M\"uller-Gronbach for helpful discussions, in particular related to questions of measurability and martingales. 

{JH is supported by the EPSRC grant EP/W026899/2 - MaThRad.}
JH and DR gratefully acknowledge support
of the DFG within project 432680300 – SFB 1456 subproject B02. K{\L} was supported by the Royal Society through the Royal Society University Research Fellowship. 
GR is partly supported by the following grants: 
EP/X028119/1, EP/R034710/1, EP/R018561 and EP/V009478/1 (EPSRC) and the OCEAN ERC synergy grant (UKRI).
Major parts of this paper were proven while JH was a visiting research fellow at the University of Warwick supported by the Royal Society and SFB 1456.

\begin{appendix}
\section{Auxiliary results}
\label{app1}
\renewcommand{\thesection}{\Alph{section}}
\setcounter{section}{1}		
The upcoming limit result follows immediately from the Corollary on page 109 of \cite{chow1960martingale} with the choice $\alpha = 2$ and $c_k = \frac{1}{\sqrt{n} (\log n)^{1/2 + \eps}}$, where $\eps>0$. 
\begin{lem}\label{lem:martingale_almost_sure_convergence}
	Let $(\xi_n)_{n \in \N}$ be a martingale (w.r.t. to some filtration $(\mathcal{G}_n)_{n \in \N}$) with differences uniformly bounded in $L^2$. 
	Then, for any $\eps>0$ we have $\mathbb{P}$-almost surely 
	\[
	\lim_{n \to \infty} \frac{1}{\sqrt{n}(\log n)^{1/2 + \eps}} \xi_n =0  .
	\]
\end{lem}

The next result shows that if $V$ is a Lyapunov function also $V^q$ is such  a function whenever $q\in (0,1)$. 
This fact is well known from the literature and we include a proof only to make the paper self contained.  

\begin{lem}\label{lem:Lyapunov_function_exponent}
	Under Assumption \ref{ass:Lyapunov_existence} for any $y \in \cY$, $\gamma \in \mathcal{I}$ and $q \in (0,1]$ we have 
	\[
	P_\gamma V^{q}(y) \leq \kappa^q V^{q}(y) +  b^q,
	\]
	with $\kappa, b$ as in Assumption \ref{ass:Lyapunov_existence}.
\end{lem}

\begin{proof}
	By Jensen's inequality holds $(P_\gamma V^{q}(y)) \leq (\kappa V(y) + b)^q$.
	Since $q\leq 1$, subadditivity\footnote{A function $\psi \colon \R \to \R $ is called subadditive if $\psi(x+y) \leq \psi(x) + \psi(y)$ for any $x,y \in \R$. In particular, subadditivity of $z \mapsto z^q$, with $q <1$ follows from §103 on p. 83 of \cite{hardy1952inequalities}.} of $z \mapsto z^q$ together with the above estimate implies
	\[
	P_\gamma V^{q}(y) \leq \left( \kappa V(y) + b \right)^{q} \leq \kappa^q V^{q}(y) + b^q.
	\] 
\end{proof}

\section{Measurability of eccentricities}\label{sec:measurability_of_E_Gamma}	
Here we provide a proof for the measurability of the $E_{\Gamma_j}(Y_k)$'s, with $j,k \in \N$. 	
For the sake of brevity fix $j,k \in \N$ and set $Y_k =Y$ as well as $\Gamma_j = \Gamma$.
We begin with showing that  
\begin{equation}\label{eq:measurability_of_E_Gamma}
	(\omega,\omega') \mapsto \int_\cY \mathds{1}_{C} ( \omega',y) \pi_{\Gamma(\omega)} (\d y) 
\end{equation}
is measurable for any $C \in \mathcal{F} \otimes \mathcal{F}_\cY =: \mathcal{G}$, considered as mapping from $(\Omega\times \Omega, \mathcal{F}\otimes \mathcal{F})$ to the reals. 
Thereto we define the following set 
\[
\mathscr{D} = \left\{ C \in   \mathcal{G} \colon  (\omega,\omega') \mapsto \int_\cY \mathds{1}_{C} ( \omega',y) \pi_{\Gamma(\omega)} (\d y) \quad \text{is measurable} \right\}. 
\]
We note that $\mathscr{D}$ is a Dynkin system and sets of the form $A \times B$, with $A \in \mathcal{F}$ and $B \in \mathcal{F}_\cY$, are contained in $\mathscr{D}$. 
Now the claimed measurability relation of \eqref{eq:measurability_of_E_Gamma} follows by the Dynkin system theorem. 

The mapping $\omega \mapsto (\omega, Y(\omega))$ is measurable\footnote{Considered as a mapping from $(\Omega, \mathcal{F})$ to $(\Omega\times\Omega, \mathcal{F}\otimes \mathcal{F})$.} so that for any $C \in \mathcal{G}$ we have measurability of 
\[
\omega \mapsto \int_\cY \mathds{1}_{C} (Y(\omega), \omega) \pi_{\Gamma(\omega)} (\d y). 
\]
Since the distance-like function $d$ is measurable (recall that it is lower semi-continuous) we deduce the measurability of $E_\Gamma(Y)$ by the usual approximation arguments with simple functions.

\section{A non-convergent AIR algorithm}\label{sec:appendix_counter_example} 
	We present an AIR method where Assumption \ref{assumption_main} is not true, and where no law of large numbers is possible. 
	It is a simple modification of a well-studied example in the adaptive MCMC literature, see e.g. \cite{andrieu2006ergodicity}.  
	Let $\gamma \in (0,1)$ and consider the transition kernel (matrix), acting on the state space $\{1,2\}$, given by
	\[
	P_\gamma = \begin{pmatrix}
		1- \gamma & \gamma\\
		\gamma & 1-\gamma
	\end{pmatrix}.
	\]
	It is easily checked that $\pi=(1/2, 1/2)$ is the invariant distribution for any $\gamma \in (0,1)$. 
	
	Choosing our distance-like function to be $d(x,y) = \mathds{1}_{\{x \neq y\}}$ implies that $\mathcal{W}$ coincides with the total variation distance.
	Straightforwardly, by the ``absolute sum" representation of $\Vert \cdot \Vert_{tv}$, one obtains
	\[
	\tau(P_\gamma) = \norm{P_\gamma(1, \cdot ) - P_\gamma(2, \cdot )}_{tv} = \frac{1}{2}\sum_{i=1}^2 
	\vert P_\gamma(1, i) - P_\gamma(2, i )\vert 
	= \vert 1-2\gamma\vert,
	\]
	such that $\tau(P^\ell_\gamma) \leq \vert 1-2\gamma\vert^\ell$. 
	We claim, what the former upper bound already indicates, that Assumption \ref{assumption_main} is not satisfied, i.e. that $\sup_{\gamma\in(0,1)}
	\tau(P^\ell_\gamma)=1$.
	Namely, for fixed $\ell \in \N$ a lower bound for $P_\gamma^\ell (1, 1)$ is the probability of ``staying at $1$'', which is simply $(1-\gamma)^\ell$. Hence $(1-\gamma)^\ell \leq P_\gamma^\ell (1, 1) \leq 1$, such that $P_\gamma^\ell (1,1) \to 1$ as $\gamma \to 0$.
	By symmetry the same arguments show $P_\gamma^\ell (2, 2) \to 1$, and thus $P_\gamma^\ell (2, 1) \to 0$, as $\gamma \to 0$.
	Now, well-known properties of the total variation distance (note that we have defined it via the Wasserstein distance here) imply 
	\begin{align*}
		\norm{P_\gamma^\ell(1, \cdot ) - P_\gamma^\ell(2, \cdot )}_{tv} & =  \sup_{A \subseteq \{1,2\}} \left\vert \sum_{x\in A}  (P_\gamma^\ell(1, x ) - P_\gamma^\ell(2, x ) )\right\vert \\
		& \geq \vert P_\gamma^\ell(1, 1 ) - P_\gamma^\ell(2, 1 ) \vert,
	\end{align*}
	such that $\sup_{\gamma \in (0,1)}\norm{P_\gamma^\ell(1, \cdot ) - P_\gamma^\ell(2, \cdot )}_{tv} = 1$ and hence $\sup_{\gamma \in (0,1)} \tau(P_\gamma^\ell) = 1$.

	Next we provide an AIR algorithm based on the family $\{P_\gamma\}_{\gamma \in (0,1)}$, where the law of large numbers is not satisfied. 
	Let $(T_\ell)_{\ell \in \N}$ be a sequence of adaptation times, with $T_\ell = \sum_{j=1}^\ell k_j$, as specified in Section \ref{sec:preliminaries} (with some arbitrary but fixed $\beta >0$).
	For each $j\in \N_0$ let $\gamma_j  \in (0,1)$ be chosen such that $(1-\gamma_j)^{k_j} =1-e^{-j-1}$ (with $k_0 :=1$).
	
	Now, let $(Y_n)_{n \in \N_0}$ correspond to an AIR algorithm initialised at $Y_0 = 1$, and afterwards using $P_{\gamma_j}$ between steps $T_j +1$ and $T_{j+1}$. 
	Note that asking for the probability of $Y_0=Y_1= \dots = Y_n = 1$ is equivalent to considering $n$ independent Bernoulli trials, with success probabilities $1-\gamma_0$ for $Y_1=1$; $1-\gamma_1$ for $Y_2=1$; $1-\gamma_2$ for $Y_3=1, \dots , Y_{T_2}=1$ etc., whence
	\[
	\mathbb{P}\left[ Y_j = 1 \text{ for }0 \leq j \leq T_\ell  \right] = \prod_{j=0}^\ell \gamma_j^{k_j} = \prod_{j=0}^\ell (1-e^{-j-1}). 
	\]
	It is not hard to verify that $\lim_{\ell\to \infty}\prod_{j=0}^\ell (1-e^{-j-1}) = \theta \in (0,1)$, such that $\mathbb{P}[X_j =1 \text{ for any }j \in \N_0] =\theta >0$. This clearly shows that the empirical averages based upon $(X_n)_{n \in \N_0}$ do not satisfy a law of large numbers in general. 
\end{appendix}

\providecommand{\bysame}{\leavevmode\hbox to3em{\hrulefill}\thinspace}
\providecommand{\MR}{\relax\ifhmode\unskip\space\fi MR }
\providecommand{\MRhref}[2]{%
	\href{http://www.ams.org/mathscinet-getitem?mr=#1}{#2}
}
\providecommand{\href}[2]{#2}

\noindent
\textbf{Author's addresses:}

\noindent
Julian Hofstadler, University of Bath, Department of Mathematical Sciences, Calverton Down, BA2 7AY Bath, United Kingdom. 

\noindent
\textit{Email:} jh4272@bath.ac.uk

\medskip
\noindent
Krzysztof {{\L}atuszy\'{n}ski}, University of Warwick, Department of Statistics, CV47AL Coventry, United Kingdom. 

\noindent
\textit{Email:} K.G.Latuszynski@warwick.ac.uk   

\medskip
\noindent
Gareth O. Roberts, University of Warwick, Department of Statistics, CV47AL Coventry, United Kingdom. 

\noindent
\textit{Email:} Gareth.O.Roberts@warwick.ac.uk

\medskip
\noindent
Daniel Rudolf, Universit\"at Passau, Faculty for Computer Science and Mathematics, Innstraße 33, 94032 Passau, Germany. 

\noindent
\textit{Email:} daniel.rudolf@uni-passau.de   

\end{document}